\documentclass[aps,11pt,twoside, nofootinbib, superscriptaddress]{revtex4}
\usepackage{amsthm}
\usepackage{amsmath,latexsym,amssymb,verbatim,enumerate,graphicx}
\usepackage{color}
\usepackage[utf8]{inputenc}
\usepackage[T1]{fontenc}
\usepackage[polish]{babel}

\newtheorem{definition}{Definition} %[section]

\newtheorem{lemma}[definition]{Lemma}

\newtheorem{thm}[definition]{Theorem}

\newtheorem{remark}[definition]{Remark}

\newtheorem*{rep@theorem}{\rep@title}
\newcommand{\newreptheorem}[2]{%
\newenvironment{rep#1}[1]{%
 \def\rep@title{#2 \ref{##1} (restatement)}%
 \begin{rep@theorem}}%
 {\end{rep@theorem}}}
\makeatother

\newreptheorem{thm}{Theorem}
\newreptheorem{lem}{Lemma}

\begin{document}
\title{Law of the Iterated Logarithm for some Markov operators}

\author{Sander C. Hille}
\affiliation{Mathematical Institute, Leiden University, P.O. Box 9512, 2300 RA Leiden, The Netherlands}
%\email{shille@math.leidenuniv.nl}
\author{Katarzyna Horbacz}
\affiliation{Institute of Mathematics, University of Silesia, 40-007 Katowice, Poland}
%\email{horbacz@math.us.edu.pl}
\author{Tomasz Szarek}
\affiliation{Institute of Mathematics, University of Gda\'{n}sk, 80-952 Gda\'{n}sk, Poland}
%\email{szarek@intertele.pl}
\author{Hanna Wojew\'odka}
\affiliation{Institute of Mathematics, University of Gda\'{n}sk, 80-952 Gda\'{n}sk, Poland}
\affiliation{National Quantum Information Center of Gda\'{n}sk, 81-824 Sopot, Poland}
\affiliation{Institute of Theoretical Physics and Astrophysics, University of Gda\'{n}sk, 80-952 Gda\'{n}sk, Poland}
%\email{hwojewod@mat.ug.edu.pl}

\thanks{T.S. was supported by the National Science Centre of Poland, grant number DEC-2012/07/B/ST1/03320}

%\date{\today}

\begin{abstract} {The Law of the Iterated Logarithm for some Markov operators, which converge exponentially to the invariant measure, is established. The operators correspond to iterated function systems which, for example, may be used to generalize the cell cycle model examined by \cite{lasotam}.}
\end{abstract}

%\subjclass[2000]{60J25 (primary), 37A25 (secondary)}
%\keywords{Law of the iterated logarithm, Markov operators}

\maketitle

\section{Introduction}
We consider some Markov operators acting on Borel measures defined on Polish spaces and corresponding to iterated function systems, which may describe e.g. the process of cell division.

One of the first cell cycle models was proposed in 1988 by J.J. Tyson and K.B. Hannsgen \cite{tyson_hannsgen}, while the full description of the research was given by A.~Murray and T.~Hunt \cite{murray_hunt}. In 1999 an~interesting result was published by A. Lasota and M.C. Mackey \cite{lasotam}. Their research was further developed by S. Hille and co-authors who proposed the generalisation of the model considered in~\cite{lasotam} and analyzed it in terms of its ergodic properties (see \cite{hille}, \cite{hhsw}), i.e. the existence of an invariant measure was established in \cite{hille}, while asymptotic stability, exponential rate of convergence to the unique invariant measure in the Fourtet-Mourier norm and the Central Limit Theorem (CLT) were proven in \cite{hhsw}. 

The aim of this paper is to verify the Law of the Iterated Logarithm (LIL), which completes the ergodic description of the generalised cell cycle model. Note that the results obtained in \cite{hhsw}, i.e. the exponential rate of convergence (see Theorem 1, \cite{hhsw}), are necessary to prove the LIL. Moreover, the variance of the normal distribution present in the thesis of the CTG (see Theorem 2, \cite{hhsw}) is~consistent with the one given in the main theorem of this paper - Theorem 2 (see Remark 1 and the proof below).

The functional form of LIL, known now as the Strassen invariant principle, was defined by V.~Strassen in 1964 \cite{strassen}. The results for martingales were further investigated in many papers (see e.g. \cite{heyde_scott}, \cite{hh} or \cite{strass}). To obtain the LIL for a~wider class of stochastic processes (i.e. for Markov processes with spectral gap in the Wasserstein metric) the martingale method due to C.C.~Heyde and D.J. Scott (Theorem 1, \cite{heyde_scott}) was used and combined with the Birkhoff individual ergodic theorem (see \cite{bms} or \cite{kom_szar}).

In this paper, however, the key role is played by the coupling measure whose construction is motivated by M. Hairer \cite{hairer}. M. Hairer proposed to build the coupling measure on the whole trajectories and use it to prove the exponential rate of convergence for some class of Markov operators (coupling measure is constructed in the same manner e.g. in \cite{sleczka} or \cite{hw}). In \cite{hhsw} we have observed that such a~coupling measure is extremely useful in the proof of the CLT. This paper shows that, in addition, it is significant to verify the LIL (see Theorem 2). 

The greatest difficulty was to prove that relevant functions are continuous. Some properties of the carefully constructed coupling measure appeared to be important in overcoming this difficulty.

The organisation of the paper goes as follows. Section 2 introduces basic notations and definitions. Most of them are adapted from \cite{billingsley}, \cite{bogachev}, \cite{rachev}, \cite{villani} or \cite{z}. Assumptions and properties of the model are stated in Section 3. We do not repeat neither the construction of the coupling measure (described in details in Sections 5-7, \cite{hhsw}), nor the proofs given in \cite{hhsw}. We restrict ourselves to recalling these facts which are necessary to prove the LIL. In the last section we finally give a~detailed proof of the LIL.

\section{Notation and basic definitions}

Let $(X,\varrho)$ be a Polish space. We denote by $B_X$ the family of all Borel subsets of $X$. 
Let $B(X)$ be the space of all bounded and measurable functions $f:X\to R$ with the supremum norm and write $C(X)$ for its subspace of all bounded and continuous functions with the supremum norm. Additionally, we consider the space $\tilde{B}(X)$ of functions $f:X\to R$ which are measurable and bounded from below.

We denote by $M(X)$ the family of all Borel measures on $X$ and by $M_{fin}(X)$ and $M_1(X)$ its subfamilies such that $\mu(X)<\infty$ and~$\mu(X)=1$, respectively. 
Elements of $M_{fin}(X)$ which satisfy $\mu(X)\leq 1$ are called sub-probability measures. To simplify notation, we write
\[\langle f,\mu\rangle =\int_X f(x)\mu(dx)\quad\text{for}\: f:X\to R,\: \mu\in M(X).\]
An operator $P:M_{fin}(X)\to M_{fin}(X)$ is called a Markov operator if
\begin{enumerate}
\item $P(\lambda_1\mu_1+\lambda_2\mu_2)=\lambda_1 P\mu_1+\lambda_2 P\mu_2\quad\text{ for }\:\lambda_1,\lambda_2\geq 0,\; \mu_1, \mu_2\in M_{fin}(X)$;
\item $P\mu(X)=\mu(X)\quad\text{ for \:$\mu\in M_{fin}(X)$}$.
\end{enumerate}
Markov operator $P$ for which there exists a linear operator $U:B(X)\to B(X)$ such that
\[\langle Uf,\mu\rangle =\langle f,P\mu\rangle \quad\text{for}\: f\in B(X),\:\mu\in M_{fin}(X)\]
is called a regular operator. Operator $U:B(X)\to B(X)$ is then called a dual operator for $P$ and it can be easily extended to $\tilde{B}(X)$. We say that a regular Markov operator is Feller if $U(C(X))\subset C(X)$. Every Markov operator $P$ may be extended to the space of signed measures on $X$ denoted by $M_{sig}(X)=\{\mu_1-\mu_2:\; \mu_1,\mu_2\in M_{fin}(X)\}$. By $\|\cdot\|$ we denote the total variation norm in $M_{sig}(X)$, i.e.
\[\|\mu\|=\mu^+(X)+\mu^-(X)\quad\text{for }\mu\in M_{sig}(X),\]
where $\mu^+$ and $\mu^-$ come from the Hahn-Jordan decomposition of $\mu$ (see \cite{halmos}). In particular, if $\mu$ is non-negative, $\|\mu\|$ is the total mass of $\mu$. 
For fixed $\bar{x}\in X$, let us introduce the function $\varrho_{\bar{x}}:X\to R$ describing the distance from the point $\bar{x}$ , i.e. $\varrho_{\bar{x}}(x)=\varrho(\bar{x},x)$ for $x\in X$. For fixed $\bar{x}\in X$ and $r>0$, we also consider the space $M_1^r(X)$ of all probability measures with finite $r$-th moment, i.e., $M_1^r(X)=\{\mu\in M_1(X):\:\int_X\varrho_{\bar{x}}^r(x)\mu(dx)<\infty\}$. The family is independent of choice of $\bar{x}\in X$. 
We call $\mu_*\in M_{fin}(X)$ an invariant measure of $P$ if $P\mu_*=\mu_*$. 
We define the support of $\mu\in M_{fin}(X)$ by
\begin{align*}
\text{supp }\mu=\{x\in X:\; \mu(B(x,r))>0\quad \text{for all }\:r>0\},
\end{align*}
where $B(x,r)$ is an open ball in $X$ with center at $x\in X$ and radius $r>0$. By $\bar{B}(x,r)$ we denote a~closed ball with center at $x\in X$ and radius $r>0$.

In $M_{sig}(X)$, we introduce the Fortet-Mourier norm
\[\|\mu\|_{\mathcal{L}}=\sup_{f\in\mathcal{L}}|\langle f,\mu\rangle |,\]
where $\mathcal{L}=\{f\in C(X):\;|f(x)-f(y)|\leq\varrho(x,y),\;|f(x)|\leq 1\;\text{ for }\:x,y\in X\}$. The space $M_1(X)$ with metric $\|\mu_1-\mu_2\|_{\mathcal{L}}$ is complete (see \cite{fortet}, \cite{rachev} or \cite{villani}).

\section{Assumptions and properties of the model}\label{sec:ass+prop}

\subsection{Assumptions}

Let $H$ be a separable Banach space. We may think of a closed subset of $H$ as a Polish space $(X,\varrho)$, where the distance $\rho$ is induced by the norm in $H$. We also condider a probability space $(\Omega,\mathcal{F},\text{\text{Prob}})$. Let $\varepsilon_*<\infty$ be given. We fix $\varepsilon\in[0,\varepsilon_*]$ and $T<\infty$. 
We consider a stochastically perturbed dynamical system of the form %, proposed in \cite{hille}
\[x_{n+1}=S(x_n,t_{n+1})+H_{n+1}\quad\text{for }n\geq 0,\]
where $(H_n)_{n\geq 1}$ is a family of independent random vectors with values in $H$ and with the same distribution $\nu^{\varepsilon}$, which is independent of $S(x_n,t_{n+1})$ and its support stays in $B(0,\varepsilon)$.
We make the following assumptions. 

\begin{itemize}
\item[(I)] We consider a sequence $(t_n)_{n\geq 1}$ of independent random variables defined on $(\Omega,\mathcal{F},\text{Prob})$ with values in $[0,T]$. Distribution of $t_{n+1}$ conditional on $x_n=x$ is given by
\begin{equation*}%\label{distr_p}
\text{Prob}(t_{n+1}<t|x_n=x)=\int_0^tp(x,u)du, \quad 0\leq t\leq T,
\end{equation*}
where $p:X\times[0,T]\to[0,\infty)$ is a~measurable and non-negative function. %such that, for every $x\in X$, $p(x,0)=0$ and $p(x,t)>0$ for $t>0$. 
In addition, $p$ is normalized, i.e. $\int_0^Tp(x,u)du=1$ for $x\in X$.

\item[(II)] Let $S:X\times[0,T]\to X$ be a~continuous function which satisfies the Lipschitz type inequality
\begin{equation*}%\label{Lip}
\varrho(S(x,t),S(y,t))\leq \lambda(x,t)\varrho(x,y)\quad\text{ for }x,y\in X, t\in[0,T],
\end{equation*}
where $\lambda:X\times[0,T]\to[0,\infty)$ is a~Borel measurable function such that 
%\begin{equation}\label{def:Lambda}
%\Lambda:=\sup_{x\in X}\int_0^T\lambda^{2+\delta}(x,t)p(x,t)dt<\frac{1}{2}.
%\end{equation}
\begin{equation}\label{def:b}
a_{2+\delta}:=\sup_{x\in X}\int_0^T\lambda^{2+\delta}(x,t)p(x,t)dt<1.
\end{equation}
Note that, due to the H\"{o}lder inequality, we also know that 
\begin{equation*}%\label{def:a}
a_1:=\sup_{x\in X}\int_0^T\lambda(x,t)p(x,t)dt\leq a_{2+\delta}^{1/(2+\delta)}<1,\quad a_2:=\sup_{x\in X}\int_0^T\lambda^2(x,t)p(x,t)dt\leq a_{2+\delta}^{2/(2+\delta)}<1.
\end{equation*}

\item[(III)] We require $\sup_{t\in[0,T]}\varrho_{\bar{x}}\Big(S(\bar{x},t)\Big)<\infty$ for some $\bar{x}\in X$ and so we can set 
\begin{align}\label{def:c}
\begin{aligned}
c:=\sup_{t\in[0,T]}\varrho_{\bar{x}}(S(\bar{x},t))+\varepsilon_*<\infty. 
\end{aligned}
\end{align}
%Let us further use notation $\varrho_{\bar{x}}(\cdot):=\varrho(x,\bar{x})$.
\item[(IV)] We assume that $p$ satisfies the Dini condition
\begin{align*}
\begin{aligned}
\int_0^T|p(x,t)-p(y,t)|dt\leq \omega(\varrho(x,y))\quad\text{for }x,y\in X,
\end{aligned}
\end{align*}
where the function $\omega:{R}_+\to{R}_+$ is non-decreasing, concave and such that $\omega(0)=0$, as~well~as
\begin{align*}
\begin{aligned}
\int_0^{\sigma}\frac{\omega(t)}{t}dt<+\infty\quad\text{for some }\sigma> 0.
\end{aligned}
\end{align*}
We can easily check that if $\zeta<1$, we have
\begin{align*}
\begin{aligned}
\varphi(t)=\sum_{n=1}^{\infty}\omega(\zeta^nt)<+\infty\quad\text{for every }t\geq 0.
\end{aligned}
\end{align*}
Moreover, $\lim_{t\to 0}\varphi(t)=0$.

\item[(V)] Function $p$ is bounded. We set $M_1:=\inf_{x\in X, t\in(0,T]}p(x,t)$, $M_2:=\sup_{x\in X, t\in[0,T]}p(x,t)$ and require $M_1>0$.

\item[(VI)] Let $\nu^{\varepsilon}$ be a Borel measure on $H$ such that its support is in $\bar{B}(0,\varepsilon)$. We set
\begin{align*}
\begin{aligned}
\nu^{\varepsilon}_x(\cdot)=\nu^{\varepsilon}(\cdot-x)\quad\text{for every $x\in X$}.
\end{aligned}
\end{align*}
We assume that $S(x,t)+h\in X$ for every $t\in[0,T]$, $x\in X$ and $h$ from the support of $\nu^{\varepsilon}$. 

\end{itemize}

The Markov chain is generated by the transition function $\Pi_{\varepsilon}:X\times B_X\to[0,1]$ of the form
\[\Pi_{\varepsilon}(x,A):=\int_0^Tp(x,t)\nu^{\varepsilon}_{S(x,t)}(A)dt.\]
Note that the function $\Pi_{\varepsilon}(\cdot,A):X\to R$ is measurable for fixed $A\in B_X$ and $\Pi_{\varepsilon}(x,\cdot):B_X\to[0,1]$ is a probability measure for $x\in X$. Hence, there exists a unique regular Markov operator $P_{\varepsilon}:M_1(X)\to M_1(X)$ which is defined as follows
\[P_{\varepsilon}\mu(A):=\int_X\Pi_{\varepsilon}(x,A)\mu(dx)\]
and its dual operator $U_{\varepsilon}:B_X\to B_X$ is given by 
\[U_{\varepsilon}f(x):=\int_X f(z)\Pi_{\varepsilon}(x,dz)\]
(see Section 1.1, \cite{z}).

\subsection{Properties of the model}\label{sub-section:properties}
Let us introduce an auxiliary model. If we fix a sequence of constants $(h_n)_{n\geq 1}\subset H$, $h_n\in\bar{B}(0,\varepsilon)$, and introduce functions $T_{h_i}(x,t):=S(x,t)+h_i$, $i\geq 1$%, which are continuous and satisfy the same Lipschitz type inequality as operator $S$ satisfies
, we may consider a stochastically perturbed dynamical system
\[\tilde{x}_{n+1}=T_{h_{n+1}}(\tilde{x}_n,t_{n+1}):=S(\tilde{x}_n,t_{n+1})+h_{n+1}\quad\text{for } n\geq 0\text{.}\]
%Let us denote
%\begin{align}\label{tilde_x}
%\begin{aligned}
%\tilde{x}_{n+1}^{x_0}:=T_{h_{n+1}}\Big(T_{h_{n}}(\ldots T_{h_2}(T_{h_1}(x_0,t_1),t_2)\ldots),t_{n+1}\Big).
%\end{aligned}
%\end{align}
Further, we define one-dimensional distributions
\begin{align}\label{constr:n+1}
\begin{aligned}
\Pi^0(x,A)=\delta_x(A) \\
\Pi^1_{h_i}(x,A)=\int_0^T1_A(T_{h_i}(x,t))p(x,t)dt \\
\ldots \\
\Pi^n_{h_1,\ldots,h_n}(x,A)=\int_X \Pi^1_{h_n}(y,A)\Pi^{n-1}_{h_1,\ldots,h_{n-1}}(x,dy),
\end{aligned}
\end{align}
where $A\in B_X$ and $\delta_x$ is a Dirac measure at $x\in X$. 
We easily obtain multidimensional distributions. Let $x\in X$ and $n\geq 0$. If we assume that $\Pi^{1,\ldots,n}_{h_1,\ldots,h_n}(x,\cdot)$ is a measure on $X^n$, generated by a~sequence $(\Pi^1_{h_i}(x,\cdot))_{i=1}^n$, then we can define the measure $\Pi_{h_1,\ldots,h_{n+1}}^{1,\ldots,n+1}(x,\cdot)$ on $X^{n+1}$ as the only measure which satisfies the condition
\begin{align}\label{constr:ndim}
\begin{aligned}
\Pi^{1,\ldots,n+1}_{h_1,\ldots,h_{n+1}}(x,A\times B)=\int_A\Pi^1_{h_{n+1}}(z_n,B)\Pi^{1,\ldots,n}_{h_1,\ldots,h_n}(x,dz)\quad\text{for }z=(z_1,\ldots,z_n), \:A\in B_{X^n},\: B\in B_X.
\end{aligned}
\end{align}
%is a measure on $X^{n+1}$. 
Finally, we obtain a~family $\{\Pi_{h_1,h_2,\ldots}^{\infty}(x,\cdot):x\in X\}$ of measures on $X^{\infty}$. 
Note that the measures $\Pi^1_{h_1}(x,\cdot),\ldots,\Pi^n_{h_1,\ldots,h_n}(x,\cdot)$, given by $(\ref{constr:n+1})$, are marginal distributions of $\Pi^{\infty}_{h_1,h_2,\ldots}(x,\cdot)$.
The existence of measure $\Pi^{\infty}_{h_1,h_2,\ldots}(x,\cdot)$ is established by the Kolmogorov theorem. More precisely, for every $x\in X$, there exists some probability space on which we can define a~stochastic process $\xi^x$ with distribution $\phi_{\xi^x}$ such that
\[\phi_{\xi^x}(B)=\text{Prob}\left(\left\{\omega\in\Omega:\xi^x(\omega)\in B\right\}\right)=\Pi_{h_1,h_2,\ldots}^{\infty}(x,B)\quad\text{ for $B\in B_{X^{\infty}}$.}\] 
Therefore, $\Pi_{h_1,h_2,\ldots}^{\infty}(x,\cdot)$ is the distribution of the non-homogeneous Markov chain $\xi^x$ on $X^{\infty}$ with sequence of transition probability functions $(\Pi^1_{h_i})_{i\in N}$  and $\phi_{\xi^x_0}=\delta_x$. %If an initial distribution is given by any $\mu\in M_{fin}(X)$, not necessarily by $\delta_x$, we define
%\[(P_{h_1,h_2,\ldots}^{\infty}{\mu})(B)=\int_X \Pi_{h_1,h_2,\ldots}^{\infty}(x,B)\mu(dx)\quad\text{ for $B\in B_{X^{\infty}}$}.\]
This construction was adapted from \cite{hairer}.

%The Markov operator $P_h:M_1(X)\to M_1(X)$, for an arbitrary $0\leq h<\varepsilon$, is given by
%\[(P_{h}\mu)(A):=\int_X\Big[\int_0^T 1_A(T_{h}(x,t))p(x,t)dt\Big]\mu(dx)\]
%and has a lot of useful properties. 
Note that, for every $n\in N$ and arbitrary $A\in B_X$, $\Pi^n_{h_1,\ldots,h_n}(\cdot,A):X\to R$ is measurable by definition. 
Furthermore, $\Pi^n_{h_1,\ldots,h_n}(x,\cdot)$ is a probability measure for $x\in X$. %, for every $x\in X$, and then $\{\Pi^n_{h_1\ldots,h_n}(x,\cdot):x\in X\}$ is a transition probability function on the $n$-th marginal. 
Again, thanks to these properties (see Section 1.1, \cite{z}), there exists a unique regular Markov operator $P^n_{h_1,\ldots,h_n}$, for which $\Pi^n_{h_1,\ldots,h_n}$ is a~transition probability function, and it is given by the formula
\[(P^n_{h_1,\ldots,h_n}\mu)(A)=\int_X\Pi^n_{h_1,\ldots,h_n}(x,A)\mu(dx)\quad\text{for }A\in B_X,\:\mu\in M_{1}(X).\]
Moreover, its dual operator $U^n_{h_1,\ldots,h_n}$ is defined as follows
\[(U^n_{h_1\ldots,h_n}f)(x)=\int_Xf(y)\Pi^n_{h_1,\ldots,h_n}(x,dy)\quad\text{for }f\in B(X).\] 
We refer the reader to \cite{hhsw}, where a lot of useful properties of $P^n_{h_1,\ldots,h_n}$ was established. Firstly, $P^n_{h_1,\ldots,h_n}$ is a Feller operator (see Remark 1, \cite{hhsw}
). Secondly, if $\mu\in M_1^i(X)$, then also $P^n_{h_1,\ldots,h_n}\mu\in M_1^i(X)$ for $i\in\{1,2\}$, which is proven in Lemmas 1 and 5 (see \cite{hhsw}). 
All estimates in proofs of these lemmas are independent of $(h_n)_{n\geq 1}$. This is crucial, because it makes all the facts valid for $P^n_{\varepsilon}$, which follows from the relation
\begin{align}\label{varepsilon_written_by_h}
\begin{aligned}
P^n_{\varepsilon}\mu(\cdot)=
\int_X\int_{\bar{B}(0,\varepsilon)}\ldots\int_{\bar{B}(0,\varepsilon)}
\Pi^n_{h_1,\ldots,h_n}(x,\cdot)\nu^{\varepsilon}(dh_1)\ldots\nu^{\varepsilon}(dh_n)\mu(dx).
\end{aligned}
\end{align}
Hence, $P_{\varepsilon}$ has the Feller property and, if $\mu\in M^i_1(X)$, then also $P_{\varepsilon}\mu\in M^i_1(X)$ for $i\in\{1,2\}$. Moreover, the dual operator $U^n_{\varepsilon}$ to $P^n_{\varepsilon}$ is of the form
\begin{align*}
\begin{aligned}
(U^n_{\varepsilon}f)(x)
=\int_{\bar{B}(0,\varepsilon)}\ldots\int_{\bar{B}(0,\varepsilon)}\int_Xf(y)\Pi^n_{h_1,\ldots,h_n}(x,dy)
\nu^{\varepsilon}(dh_1)\ldots\nu^{\varepsilon}(dh_n)\quad\text{for }f\in B(X)
\end{aligned}
\end{align*}
and it may be extended to $\tilde{B}(X)$.

In Section 7 of \cite{hhsw} we adapt the construction introduced in \cite{hairer} and, for some fixed $x_0,y_0\in X$ and initial distribution $\delta_{(x_0,y_0,1)}$, we build an appropriate coupling measure $\hat{C}^{\infty}_{h_1,h_2,\ldots}((x_0,y_0,1),\cdot)$ on $(X^2\times\{0,1\})^{\infty}$, which has the following properties
\begin{itemize}
\item[(a)] $\Pi^*_{(X^2)^{\infty}}\hat{C}^{\infty}_{h_1,h_2,\ldots}((x_0,y_0,1),\cdot)={C}^{\infty}_{h_1,h_2,\ldots}((x_0,y_0),\cdot)$, 
where $\Pi^*_{(X^2)^{\infty}}:(X^2\times\{0,1\})^{\infty}\to (X^2)^{\infty}$ is the projection on $(X^2)^{\infty}$,
\item[(b)] ${C}^{\infty}_{h_1,h_2,\ldots}((x_0,y_0),A\times X)=\Pi_{h_1,h_2,\ldots}^{\infty}(x_0,A)$ and ${C}^{\infty}_{h_1,h_2,\ldots}((x_0,y_0),X\times B)=\Pi_{h_1,h_2,\ldots}^{\infty}(y_0,B)$ for $A,B\in \otimes_{i=1}^{\infty}B_X$,
\item[(c)] the marginals $C^n_{h_1,\ldots,h_n}((x_0,y_0),\cdot)$ of ${C}^{\infty}_{h_1,h_2,\ldots}((x_0,y_0),\cdot)$ are coupling measures on $X^2$, i.e. they couple measures $\Pi^n_{h_1,\ldots,h_n}(x_0,\cdot)$, $\Pi^n_{h_1,\ldots,h_n}(y_0,\cdot)$, given by (\ref{constr:n+1}),
\item[(d)] marginal coupling measures are related by the condition 
\begin{align}\label{relation_marginal_coupling}
C^n_{h_1,\ldots,h_n}((x_0,y_0),\cdot)
=\int_XC^1_{h_n}((z_1,z_2),\cdot)C^{n-1}_{h_1,\ldots,h_{n-1}}((x_0,y_0),dz_1\times dz_2),
\end{align}
which follows from the construction of the coupling measure on the whole trajectories (see Section 7, \cite{hhsw}).
\end{itemize}
Let $g\in B(X)$ be a Lipschitz continuous function with constant $L_g>0$. Then, it follows from Lemma~4 and Remark~2 \cite{hhsw} that there exist $q\in(0,1)$ and ${C}>0$ such that
\begin{align}\label{coupling}
\begin{aligned}
\int_{X^2}|g(u)-g(v)|(\Pi_{X^2}^*\Pi^*_n\hat{C}^{\infty}_{h_1,h_2\ldots}((x,y,1),\cdot))(du\times dv)\leq Gq^n{C}(1+\varrho_{\bar{x}}(x)+\varrho_{\bar{x}}(y)), x,y\in X, n\in N,
\end{aligned}
\end{align}
where $\Pi^*_n:(X^2\times\{0,1\})^{\infty}\to X^2\times\{0,1\}$ is the projection on the $n$-th component, $\Pi^*_{X^2}:X^2\times\{0,1\}\to X^2$ is the projection on $X^2$ and $G:=\max\{L_g,\sup_{x\in X}|g(x)|\}$.
The above inequality is crucial in the proofs of the exponential rate of convergence 
%\begin{align}\label{SG}
%\begin{aligned}
%\|P_{\varepsilon}^n\mu-\mu_*\|_{\mathcal{L}}\leq Cq^n,
%\end{aligned}
%\end{align}
%for $n\geq 0$ and some constants $C=C(\mu)>0$ and $q\in[0,1)$,
 and the CLT (see Theorems~1~and~2,~\cite{hhsw}).

%Let us now establish additional fact, which should be useful throughout the paper.

Let us introduce some additional notation. Let $(x_n)_{n\geq 0}$ be a Markov chain. For a given probability measure $\mu\in M_{fin}(X)$ and a~Borel event $B\in \otimes_{i=1}^{\infty}B_X$, we write
\begin{align*}
\begin{aligned}
\text{Prob}_{\mu}(B):=\int_X\text{Prob}((x_0,x_1,\ldots)\in B|x_0=x)\mu(dx).
\end{aligned}
\end{align*}
Moreover,
\begin{align*}
\begin{aligned}
&\text{Prob}_{\mu}(x_0\in A_0,x_1\in A_1,\ldots,x_n\in A_n)\\
&=\int_{A_0}\int_{A_1}\ldots\int_{A_{n-1}}\Pi_{\varepsilon}(s_{n-1},A_n)\Pi_{\varepsilon}(s_{n-1},ds_{n-1})\ldots\Pi_{\varepsilon}(s_0,ds_1)\mu(ds_0)
\end{aligned}
\end{align*}
for $n\geq 0$ and $A_0,\ldots,A_n\in B_X$ (compare with Theorem 3.4.1, \cite{tweedie}). 
The respective expectation is denoted by $E_{\mu}$. 
For $\mu=\delta_x$, we just write $\text{Prob}_{x}$ and $E_x$. 

\begin{lemma}\label{H3}
Let $a_1$, $a_2$, $a_{2+\delta}$ be given as in Assumption~(II) and let $c$ be given as in Assumption~(III). If $\mu\in M_1^{j}(X)$, then also $P^n_{\varepsilon}\mu\in M_1^{j}(X)$ for $n\geq 1$ and $j\in \{1,2,2+\delta\}$, i.e. %there are $\bar{x}\in X$ and $\delta>0$ such that 
\[\sup_{n\geq 0}E_{\mu}\left(\varrho_{\bar{x}}^{j}(x_n)\right)=\sup_{n\geq 0}\int_X\varrho_{\bar{x}}^{j}(x)P^n_{\varepsilon}\mu(dx)<\infty,\]
which in stationary case means that $\varrho_{\bar{x}}^{j}\in L^1(\mu_*)$ for $j\in\{1,2,2+\delta\}$.
\end{lemma}

\begin{proof}[Proof]
Let $\mu\in M_1^{j}(X)$ for $j\in\{1,2,2+\delta\}$ and let $h\in \bar{B}(0,\varepsilon)$. Note that
\begin{align*}
\begin{aligned}
\left(\left\langle \varrho_{\bar{x}}^{j},P_h\mu\right\rangle\right)^{1/j}
&=\left(\int_X\int_X\varrho_{\bar{x}}^{j}(y)\Pi_h(x,dy)\mu(dx)\right)^{1/j}\\
&=\left(\int_X\int_0^T \varrho_{\bar{x}}^{j}(T_h(x,t))p(x,t)\,dt\,\mu(dx)\right)^{1/j}\leq\|\varrho_{\bar{x}}\circ T_h\|_{{L}^{j}(\varsigma)},
\end{aligned} 
\end{align*}
where $\| \cdot \|_{L^{j}(\varsigma)}$ is the norm in the space $L^{j}(\varsigma)$ such that
$\|f\|_{L^{j}(\varsigma)}=\left|\left\langle f^{j},\varsigma\right\rangle\right|^{1/j}$ for $f\in \tilde{B}(X\times[0,T])$ and $\varsigma\in M_{fin}(X\times[0,T])$ given by 
$\varsigma(A):=\int_{X\times[0,T]} 1_{A}((x,t))p(x,t)\,dt\,\mu(dx)$ for $A\in B_X\otimes B_{[0,T]}$. 
By Assumptions (I) and (II) we obtain
\begin{align*}
\begin{aligned}
(\varrho_{\bar{x}}\circ T_h)(x,t)=\varrho(T_h(x,t),\bar{x})\leq \varrho(T_h(x,t),T_h(\bar{x},t))+\varrho(T_h(\bar{x},t),\bar{x})\leq\lambda(x,t)\varrho_{\bar{x}}(x)+c
\end{aligned}
\end{align*}
and therefore 
\begin{align*}
\begin{aligned}
%\left|\left\langle \varrho_{\bar{x}}^{j},P_h\mu\right\rangle\right|^{1/j}
\|\varrho_{\bar{x}}\circ T_h\|_{L^{j}(\varsigma)}
\leq\left|\int_{X\times[0,T]}\lambda^{j}(x,t)\varrho_{\bar{x}}^{j}(x)p(x,t)\,dt\,\mu(dx)\right|^{1/j}+c
&\leq\left|a_j\left\langle \varrho_{\bar{x}}^{j},\mu\right\rangle\right|^{1/j} +c,
\end{aligned}
\end{align*}
which finally gives us
\begin{align*}
\begin{aligned}
\left\langle \varrho_{\bar{x}}^{j},P^n_{h_1,\ldots,h_n}\mu\right\rangle
&\leq\left(a_j^{1/j}\left|\left\langle \varrho_{\bar{x}}^{j},P^{n-1}_{h_1,\ldots,h_{n-1}}\mu\right\rangle\right|^{1/j} +c\right)^{j}\\
&\leq \left(a_j^{2/j}\left|\left\langle \varrho_{\bar{x}}^{j},P^{n-2}_{h_1,\ldots,h_{n-2}}\mu\right\rangle\right|^{1/j} +c\left(1+a_j^{1/j}\right)\right)^{j}\\
&\leq\ldots
\leq\left(a_j^{n/j}\left|\left\langle \varrho_{\bar{x}}^{j},\mu\right\rangle\right|^{1/j}+{c}{\left(1-a_j^{1/j}\right)^{-1}}\right)^{j},
\end{aligned}
\end{align*}
where $a_j^{1/j}<1$, $c<\infty$ by assumption and the estimations are independent of $(h_i)_{i\geq 1}$. Hence, we obtain
\begin{align}\label{estim_2+delta_epsilon}
\left\langle \varrho_{\bar{x}}^{j}, P^n_{\varepsilon}\mu\right\rangle \leq \left(a_j^{n/j}\left|\left\langle \varrho_{\bar{x}}^{j},\mu\right\rangle\right|^{1/j}+{c}{\left(1-a_j^{1/j}\right)^{-1}}\right)^{j}<\infty,
\end{align}
which completes the proof.
\end{proof}

%\begin{remark}\label{rem:2+delta}
%For every $\xi, \zeta\geq 0$ and $\epsilon>0$, we may choose $\delta_1>0$ such that, for $\delta<\delta_1$,
%\begin{align}\label{2+delta}
%\begin{aligned}
%(\xi+\zeta)^{2+\delta}\leq (2+\epsilon)\xi^{2+\delta} + (2+\epsilon)\zeta^{2+\delta} .
%\end{aligned}
%\end{align}
%\end{remark}

%\begin{remark}\label{rem:estim_finite_second_moment}
%Repeating the proof of Lemma 1, we may show that 
%\begin{align}
%\langle \varrho_{\bar{x}}^{2}, P^n_{\varepsilon}\mu\rangle 
%\leq \Big(b^{n/(2+\delta)}\Big(\langle \varrho_{\bar{x}}^{2},\mu\rangle\Big)^{1/2}+{c}{\Big(1-b^{1/(2+\delta)}\Big)^{-1}}\Big)^{2}<\infty.
%\end{align}
%\end{remark}

\section{The law of the iterated logarithm applied to Markov chains}

\subsection{A martingale result}

We begin with presenting a classical result established in \cite{heyde_scott}. Let $(M_n)_{n\geq 0}$, defined on $(\Omega, \mathcal{F}, \text{Prob})$, be a martingale with respect to $(\mathcal{F}_n)_{n\geq 0}$, where $\mathcal{F}_0=\{\Omega,\emptyset\}$ and $\mathcal{F}_n$ is the $\sigma$-field generated by $M_1,M_2,\ldots,M_n$ for $n>0$. We call $(\mathcal{F}_n)_{n\geq 0}$ the natural filtration of $(M_n)_{n\geq 0}$. Let us define $(Z_n)_{n\geq 0}$ such that $Z_0=M_0=0$ $\text{Prob}$-a.s. %wyjanij to pojêcie
and $Z_n=M_n-M_{n-1}$ for $n\geq 1$. Further, let $s_n^2:=EM_n^2<\infty$.

We consider the metric space $(\mathcal{C},\tilde{\varrho})$ of all real-valued continuous functions on $[0,1]$ with 
\begin{align*}
\begin{aligned}
\tilde{\varrho}(f_1,f_2)=\sup_{t\in[0,1]}|f_1(t)-f_2(t)|\quad \text{for }f_1,f_2\in \mathcal{C}.
\end{aligned}
\end{align*}
Then, we define the set $\mathcal{K}$ of all absolutely continuous functions $f\in \mathcal{C}$ such that $f(0)=0$ and $\int_0^1(f'(t))^2dt\leq 1$. The real function $F$ on $[0,\infty)$ is given by $F(s)=\sup\{n:s_n^2\leq s\}$, while the~sequence of real random functions $(\eta_n)_{n\geq 1}$ on $[0,1]$ is of the form
\begin{align*}
\begin{aligned}
\eta_n(t)=\frac{M_k+(s_n^2t-s_k^2)(s_{k+1}^2-s_k^2)^{-1}Z_{k+1}}{\sqrt{2s_n^2\log\log s_n^2}}%mo¿e wyjanienie, ¿e log, to jaki logarytm
\end{aligned}
\end{align*}
for $n>F(e)$, where $1\leq k\leq n-1$, $\;s_k^2\leq s_n^2 t\leq s_{k+1}^2$. We put $\eta_n(t)=0$ for $n\leq F(e)$.

\begin{thm}[Theorem 1, \cite{heyde_scott}]
If $s_n^2\to\infty$ and the following conditions are fulfilled
\begin{align}\label{th1}
\begin{aligned}
&\sum_{n=1}^{\infty}s_n^{-4}E\Big(Z_n^41_{\{|Z_n|<\gamma s_n\}}\Big)< \infty\quad\text{for some }\gamma>0,
\end{aligned}
\end{align}
\begin{align}\label{th2}
\begin{aligned}
\sum_{n=1}^{\infty}s_n^{-1}E\Big(|Z_n|1_{\{|Z_n|\geq\vartheta s_n\}}\Big)<\infty\quad\text{for all }\vartheta>0,
\end{aligned}
\end{align}
\begin{align}\label{th3}
\begin{aligned}
s_n^{-2}\sum_{k=1}^nZ_k^2\to 1\quad \text{Prob}\text{-a.s., as }\;n\to\infty,
\end{aligned}
\end{align}
then $(\eta_n)_{n> F(e)}$ is relatively compact in $\mathcal{C}$ and the set of its limit points coincides with $\mathcal{K}$.
\end{thm}

\subsection{Application to the model}
We consider the model initially introduced in \cite{hille} and so Assumptions (I)-(VI) are fulfilled. 
Let us consider Markov chains $(x_n)_{n\geq 0}$, $(y_n)_{n\geq 0}$ with state space $X$, transition probability function $\Pi_{\varepsilon}$ and initial distributions $\mu,\nu\in M_1^{2+\delta}(X)$, respectively. By $\mu_*$ we denote an invariant measure for the model, which exists due to Theorem 1 in \cite{hhsw}.

%Recall that the following conditions are fulfilled in our model, initially introduced in \cite{hille}.
%\begin{itemize}
%\item[(H0)] The Markov operator $P_{\varepsilon}$ has the Feller property.
%\item[(H1)] If $\mu\in M^i_1(X)$, then also $P_{\varepsilon}\in M^i_1(X)$, $i\in \{1,2\}$.
%\item[(H2)] There exist $C>0$ and $q\in[0,1)$ such that, for $\mu_1,\mu_2\in M_1^1(X)$,
%\[\|P_{\varepsilon}^n\mu_1-P^n_{\varepsilon}\mu_2\|_{\mathcal{L}}\leq Cq^n(1+\langle \varrho_{\bar{x}},\mu_1\rangle +\langle \varrho_{\bar{x}},\mu_2\rangle )\quad \text{ for }n\geq 0\]
%(compare with the spectral gap property established in Theorem 1, \cite{hhsw})
%\item[(H3)] There are $\bar{x}\in X$ and $\delta>0$ such that 
%\[\sup_{n\geq 0}E_{\mu}\varrho_{\bar{x}}^{2+\delta}(x_n)=\sup_{n\geq 0}\int_X\varrho_{\bar{x}}^{2+\delta}(x)P^n_{\varepsilon}\mu(dx)<\infty,\]
%which in stationary case means that $\varrho_{\bar{x}}^{2+\delta}\in L_1(\mu_*)$ (see Lemma \ref{H3} for proof).
%\end{itemize}

Further, let $g\in B(X)$ be a Lipschitz function with constant $L_g>0$. It is also assumed that $\langle g,\mu_*\rangle =0$ (otherwise we could consider $\tilde{g}=g-\langle g,\mu_*\rangle $).

Let $n\geq 0$. Note that by the Minkowski inequality in $L^{2+\delta}(P^n_{\varepsilon}\mu)$ and Lemma \ref{H3} (precisely estimation (\ref{estim_2+delta_epsilon})), we obtain
\begin{align}\label{minkowski}
\begin{aligned}
\Big(E_{\mu}\Big(|g(x_n)|^{2+\delta}\Big)\Big)^{1/(2+\delta)}&=\Big(\int_X|g(x)|^{2+\delta}P^n_{\varepsilon}\mu(dx)\Big)^{1/(2+\delta)}\\
&\leq|g(\bar{x})|+L_g\Big(\int_X\varrho_{\bar{x}}^{2+\delta}(x)P^n_{\varepsilon}\mu(dx)\Big)^{1/(2+\delta)}\\
&\leq |g(\bar{x})|+L_g\Big( \Big(\langle \varrho_{\bar{x}}^{2+\delta},\mu\rangle\Big)^{1/(2+\delta)}+{c}{\Big(1-b^{1/(2+\delta)}\Big)^{-1}}\Big)
<\infty
\end{aligned}
\end{align}
and consequently $\sup_{n\geq 0}E_{\mu}\Big(|g(x_n)|^{2+\delta}\Big)<\infty$. Let $x\in X$. By (\ref{varepsilon_written_by_h}) we have
\begin{align}\label{finite_chi}
\begin{aligned}
\sum_{i=0}^{\infty}|U^i_{\varepsilon}g(x)|&=\sum_{i=0}^{\infty}|\langle g,P^i_{\varepsilon}\delta_x\rangle -\langle g,P^i_{\varepsilon}\mu_*\rangle |\\
&\leq\sum_{i=0}^{\infty}\int_{(\bar{B}(0,\varepsilon))^i}|\langle g,P^i_{h_1,\ldots,h_i}\delta_x\rangle -\langle g,P^i_{h_1,\ldots,h_i}\mu_*\rangle |\nu^{\varepsilon}(dh_1)\ldots\nu^{\varepsilon}(dh_i).
\end{aligned}
\end{align}
Further, due to (\ref{coupling}), we obtain
\begin{align}\label{finite_chi_cd}
\begin{aligned}
|\langle g,P^i_{h_1,\ldots,h_i}\delta_x\rangle -\langle g,P^i_{h_1,\ldots,h_i}\mu_*\rangle |&\leq
\int_X\int_{X^2}|g(u)-g(v)|
(\Pi^*_{X^2}\Pi^*_i\hat{C}^{\infty}_{h_1,h_2,\ldots}((x,y,1),\cdot))
(du\times dv)\;\mu_*(dy)\\
&\leq q^iGC(1+\varrho_{\bar{x}}(x)+\langle \varrho_{\bar{x}},\mu_*\rangle ),
\end{aligned}
\end{align}
where $G:=\max\{L_g,\sup_{x\in X}|g(x)|\}$. 
Comparing (\ref{finite_chi}) and (\ref{finite_chi_cd}), we easily obtain
\begin{align}\label{fact:finite_chi}
\begin{aligned}
\sum_{i=0}^{\infty}|U^i_{\varepsilon}g(x)|\leq (1-q)^{-1}GC(1+\varrho_{\bar{x}}(x)+\langle \varrho_{\bar{x}},\mu_*\rangle )<\infty
\end{aligned}
\end{align}
and therefore we may define the function 
\begin{align}\label{def:chi}
\begin{aligned}
\chi(x):=\sum_{i=0}^{\infty}U^i_{\varepsilon}g(x)\quad \text{for }x\in X.
\end{aligned}
\end{align}

\begin{lemma}\label{lem:chi-chi}
Let us consider the function $\chi$, defined above. We have
\begin{align*}
\begin{aligned}
|\chi(x)-\chi(y)|\leq \frac{GC}{1-q}(1+\varrho_{\bar{x}}(x)+\varrho_{\bar{x}}(y))\quad\text{for }x,y\in X.
\end{aligned}
\end{align*}
\end{lemma}
\begin{proof}[Proof]
Fix $x,y\in X$. 
Following $(\ref{varepsilon_written_by_h})$ and $(\ref{coupling})$, we obtain
\begin{align*}%\label{chi-chi}
\begin{aligned}
&|\chi(x)-\chi(y)|\\&=\Big|\sum_{i=0}^{\infty}U^i_{\varepsilon}g(x)-\sum_{i=0}^{\infty}U^i_{\varepsilon}g(y)\Big|
\leq\sum_{i=0}^{\infty}\Big|\langle g,P^i_{\varepsilon}\delta_x\rangle -\langle g,P^i_{\varepsilon}\delta_y\rangle \Big|\\
&=\sum_{i=0}^{\infty}\int_{(\bar{B}(0,\varepsilon))^i}\int_{X^2}|g(u)-g(v)|
(\Pi^*_{X^2}\Pi^*_i\hat{C}^{\infty}_{h_1,h_2,\ldots}((x,y,1),\cdot))
(du\times dv)\;\nu^{\varepsilon}(dh_1)\ldots\nu^{\varepsilon}(dh_i)\\
&\leq \sum_{i=0}^{\infty}q^iGC(1+\varrho_{\bar{x}}(x)+\varrho_{\bar{x}}(y))
=(1-q)^{-1}GC(1+\varrho_{\bar{x}}(x)+\varrho_{\bar{x}}(y)).
\end{aligned}
\end{align*}
\end{proof}

Further, let us introduce random variables
\begin{align}\label{def:M_n}
\begin{aligned}
M_n:=\chi(x_n)-\chi(x_0)+\sum_{i=0}^{n-1}g(x_i)\quad \text{for }n\geq 0
\end{aligned}
\end{align}
and their square integrable differences which are of the form
\begin{align}\label{def:Z_n}
\begin{aligned}
Z_n=\chi(x_n)-\chi(x_{n-1})+g(x_{n-1})\quad\text{for }n\geq 1
\end{aligned}
\end{align}
and
\begin{align*}
\begin{aligned}
Z_0=0\quad\text{Prob-a.s.}
\end{aligned}
\end{align*}

\begin{lemma}
$(M_n)_{n\geq 0}$, defined by (\ref{def:M_n}), is a martingale on the space $(X^{\infty},\otimes_{i=1}^{\infty}B_X, \text{Prob}_\mu)$ with respect to its natural filtration.
\end{lemma}

\begin{proof}[Proof]
Note that by the Markov property we have
\begin{align}\label{Markov_PIL}
%E_{\mu}\Big(g(x_{n+1})|\mathcal{F}_n\Big)=U_{\varepsilon}g(x_n)
E_{\mu}\left(g(x_{n+1})|\mathcal{F}_n\right)(\omega)
=E_{x_n(\omega)}\left(g\right)=(U_{\varepsilon}g)(x_n(\omega))
\end{align}
and therefore
\begin{align*}
\begin{aligned}
E_{\mu}\left(M_{n+1}|\mathcal{F}_n\right)&=E_{\mu}\left(\chi(x_{n+1})|\mathcal{F}_n\right)-\chi(x_0)+\sum_{i=0}^ng(x_i)\\
&=\sum_{i=0}^{\infty}U_{\varepsilon}(U^i_{\varepsilon}g)(x_n)-\chi(x_0)+\sum_{i=0}^ng(x_i)\\
&=\sum_{i=1}^{\infty}U^i_{\varepsilon}g(x_n)+U_{\varepsilon}^0g(x_n)-\chi(x_0)+\sum_{i=0}^{n-1}g(x_i)=M_n.
\end{aligned}
\end{align*}
\end{proof}

\begin{lemma}\label{prop:Z_n}
The square integrable differences $(Z_n)_{n\geq 1}$, given by (\ref{def:Z_n}), are such that $E_{\mu_*}Z_1^2<\infty$.
\end{lemma}

\begin{proof}[Proof]
Let $n\geq 1$ and $\mu\in M_1^{2+\delta}(X)$. Note that, by the Markov property (see (\ref{Markov_PIL}), we obtain
\begin{align}\label{estim:Z}
\begin{aligned}
E_{\mu}(Z_{n+1}^2)=E_{P^{n}_{\varepsilon}\mu}(Z_1^2)&=\int_XE\left(\left(\chi(x_1)-\chi(x_0)+g(x_0)\right)^2|x_0=x\right)P_{\varepsilon}^n\mu(dx)\\
&\leq\int_XE\left(2\chi^2(x_1)+2(\chi-g)^2(x_0)|x_0=x\right)\:P_{\varepsilon}^n\mu(dx)\\
&= \int_X2E\left(\chi^2(x_1)|x_0=x\right)+2(\chi-g)^2(x)\:P_{\varepsilon}^n\mu(dx)\\
&\leq 2\int_X\left(U_{\varepsilon}\chi^2\right)(x)P^n_{\varepsilon}\mu(dx)
+4\int_X\chi^2(x)P^n_{\varepsilon}\mu(dx)+4\int_Xg^2(x)P^n_{\varepsilon}\mu(dx)\\
&=2\int_X\chi^2(x)P^{n+1}_{\varepsilon}\mu(dx)+4\int_X\chi^2(x)P^n_{\varepsilon}\mu(dx)+4E_{\mu}(g(x_n))^2.
\end{aligned}
\end{align}
Following $(\ref{minkowski})$, we easily obtain that the last component of (\ref{estim:Z}) is finite. Now, it is enough to~establish that $\left\langle \chi^2,P^n_{\varepsilon}\mu\right\rangle $ is finite. Note that,
\begin{align}\label{estim:chi}
\begin{aligned}
\int_X\chi^2(x)P^n_{\varepsilon}\mu(dx)&=\int_X\left(\left(\chi(x)-\chi(\bar{x})\right)+\chi(\bar{x})\right)^2P^n_{\varepsilon}\mu(dx)\\
&\leq 2\chi^2(\bar{x})+2\int_X\left(\chi(x)-\chi(\bar{x})\right)^2P^n_{\varepsilon}\mu(dx).
\end{aligned}
\end{align}
The first component of (\ref{estim:chi}) is finite due to (\ref{fact:finite_chi}) and (\ref{def:chi}). To show finiteness of the second component, let us refer to Lemma \ref{lem:chi-chi} to obtain
\begin{align}\label{estim:chi2}
\begin{aligned}
2\int_X(\chi(x)-\chi(\bar{x}))^2P^n_{\varepsilon}\mu(dx)
&\leq 2\int_X(1-q)^{-2}G^2C^2(1+\varrho_{\bar{x}}(x))^2P^n_{\varepsilon}\mu(dx)\\
&\leq 4G^2C^2(1-q)^{-2}\left(1+\left\langle \varrho^2_{\bar{x}},P^n_{\varepsilon}\mu\right\rangle \right).
\end{aligned}
\end{align}
Consequently, by (\ref{minkowski}) and (\ref{estim:Z})-(\ref{estim:chi2}) we have
\begin{align*}
\begin{aligned}
E_{P^{n}_{\varepsilon}\mu}(Z_1^2)<\tilde{C}
\left(1+\left\langle \varrho_{\bar{x}}^2,P^{n+1}_{\varepsilon}\mu\right\rangle +
\left\langle \varrho_{\bar{x}}^2,P^n_{\varepsilon}\mu\right\rangle\right) 
\end{aligned}
\end{align*}
and therefore, according to Lemma \ref{H3}, we obtain
\begin{align}\label{supEZn-finite}
\begin{aligned}
\sup_{n\geq 0}E_{P^{n}_{\varepsilon}\mu}(Z_1^2)
\leq\bar{C}\left(1+\left\langle \varrho_{\bar{x}}^2,\mu\right\rangle\right)<\infty.
\end{aligned}
\end{align}
Now, we easily check that $x\mapsto E_x(Z_1^2\wedge k)$ is a bounded countinuous function, for every $k\geq 1$. Hence, we have $\lim_{n\to\infty}E_{P^n_{\varepsilon}\mu}(Z_1^2\wedge k)=E_{\mu_*}(Z_1^2\wedge k)$, as $n\to\infty$. By (\ref{supEZn-finite}), $(E_{\mu_*}(Z_1^2\wedge k))_{k\geq 1}$ is bounded. As a~consequence, if we apply the Monotone Convergence Theorem, we finally obtain $\lim_{k\to\infty}E_{\mu_*}(Z_1^2\wedge k)=E_{\mu_*}(Z_1^2)<\infty$.
\end{proof}

Set
\begin{align}\label{def:sigma2}
\begin{aligned}
\sigma^2:=E_{\mu_*}Z_1^2.
\end{aligned}
\end{align}

\begin{lemma}\label{lem:sn/n=sigma}
Let $\mu\in M_1^{2+\delta}(X)$. If, for every $n\geq 0$, $M_n$ is given by $(\ref{def:M_n})$ and $s^2_n=E_{\mu}M_n^2<\infty$, then
\[\lim_{n\to\infty}\frac{s_n^2}{n}=\sigma^2.\]
\end{lemma}

\begin{proof}[Proof]
Following the proof of Lemma \ref{prop:Z_n} (inequalities (\ref{estim:Z})-(\ref{supEZn-finite})), we obtain
\begin{align}\label{supEZn-delta-finite}
\begin{aligned}
\sup_{n\geq 1}E_{\mu}|Z_n|^{2+\delta}<\infty.
\end{aligned}
\end{align}
Therefore,
\begin{align*}
\begin{aligned}
\sup_{n\geq 1}E_{\mu}\left(Z_n^21_{\{|Z_n|^2\geq k\}}\right)
\leq\sup_{n\geq 1}E_{\mu}\left(|Z_n|^{2+\delta}(|Z_n|^2)^{-\delta/2} 1_{\{|Z_n|^2\geq k\}}\right)
\leq k^{-\delta/2}\sup_{n\geq 1}E_{\mu}|Z_n|^{2+\delta}\to 0,
\end{aligned}
\end{align*}
as $k\to\infty$. Now, since $(Z_1^2\wedge k)$ are bounded continuous and $P_{\varepsilon}$ is Feller, we obtain 
\[\lim_{n\to\infty}E_{P^n_{\varepsilon}\mu}(Z_1^2\wedge k)=E_{\mu_*}(Z_1^2\wedge k) \quad\text{for every }k\geq 1.\]
Note that the sequence $(E_{\mu_*}(Z_1^2\wedge k))_{k\geq 1}$ is bounded and therefore the Monotone Convergence Theorem implies
\[\lim_{k\to\infty}E_{\mu_*}(Z_1^2\wedge k)=E_{\mu_*}Z_1^2=\sigma^2.\]
Hence, we also have
\begin{align*}
\lim_{n\to\infty}E_{\mu}Z_{n+1}^2=\lim_{n\to\infty}E_{P^n_{\varepsilon}\mu}Z_1^2=E_{\mu_*}Z_1^2=\sigma^2.
\end{align*}
Finally, by orthogonality of martingale differences, we obtain
\begin{align*}
\begin{aligned}
\lim_{n\to\infty}\frac{s_n^2}{n}=\lim_{n\to\infty}\frac{E_{\mu}M_n^2}{n}=\lim_{n\to\infty}\frac{\sum_{i=1}^nE_{\mu}Z_i^2}{n}=\sigma^2,
\end{aligned}
\end{align*}
which completes the proof.
\end{proof}

\begin{remark}
The variance $\sigma^2=E_{\mu_*}Z_1^2$ is compatible with the variance of limiting normal distribution in the CLT (see Theorem 2, \cite{hhsw}), i.e. with $\sigma^2=\lim_{n\to\infty}E_{\mu_*}\left(\left(S_n^*\right)^2\right)$, where $S_n^*=n^{-1/2}(g(x_0)+\ldots +g(x_{n-1}))$ for $n\in N$.
%In \cite{woodr}, which is the paper we follow in our proof of the CLT (see \cite{hhsw}), the variance is exactly given by $\lim_{n\to\infty}s_n^2/n$.
\end{remark}
\begin{proof}[Proof]
Note that 
\begin{align}\label{1}
\begin{aligned}
\lim_{n\to\infty} E_{\mu_*}\left(\left(S_n^*\right)^2\right)
&=\lim_{n\to\infty} E_{\mu_*}\left(n^{-1}\left(\sum_{i=0}^{n-1}g(x_i)\right)^2\right)\\
&=\lim_{n\to\infty} E_{\mu_*}\left(n^{-1}\left(M_n+\chi(x_0)-\chi(x_n)\right)^2\right)\\
%&=\lim_{n\to\infty} E_{\mu_*}\left(\frac{1}{n}M_n
%+\frac{1}{n}M_n\left(\chi(x_0)-\chi(x_n)\right)
%+\frac{1}{n}\left(\chi(x_0)-\chi(x_n)\right)^2\right)\\
&=\lim_{n\to\infty} E_{\mu_*}\left(n^{-1}M_n^2\right)
+\lim_{n\to\infty} 2n^{-1/2}E_{\mu_*}\left(\left(n^{-1/2}M_n\right)
\left(\chi(x_0)-\chi(x_n)\right)\right)\\
&+\lim_{n\to\infty} n^{-1}E_{\mu_*}\left(\chi(x_0)-\chi(x_n)\right)^2.
\end{aligned}
\end{align}
Referring to Lemma \ref{lem:sn/n=sigma}, we have $\lim_{n\to\infty} E_{\mu_*}\left(n^{-1}M_n^2\right)=E_{\mu_*}Z_1^2$. Further, due to Lemma \ref{lem:chi-chi} we obtain
\begin{align}\label{2}
\begin{aligned}
E_{\mu_*}\left(\chi(x_0)-\chi(x_n)\right)^2
&=\int_X\int_X\left(\chi(u)-\chi(v)\right)^2P^n_{\varepsilon}\delta_u(dv)\mu_*(du)\\
&\leq\int_X\int_XG^2C_5^2(1-q)^{-2}(1+\varrho_{\bar{x}}(u)+\varrho_{\bar{x}}(v))^2P^n_{\varepsilon}\delta_u(dv)\mu_*(du)\\
&\leq C_0\int_X\int_X(1+\varrho_{\bar{x}}^2(u)+\varrho_{\bar{x}}^2(v))P^n_{\varepsilon}\delta_u(dv)\mu_*(du)\\
&\leq C_0\int_X \left(1+\varrho_{\bar{x}}^2(u)+\left\langle \varrho_{\bar{x}}^2,P^n_{\varepsilon}\delta_u\right\rangle\right)\mu_*(du),
\end{aligned}
\end{align}
where $C_0$ is some constant. According to  (\ref{estim_2+delta_epsilon}) we obtain
\begin{align}\label{3}
E_{\mu_*}\left(\chi(x_0)-\chi(x_n)\right)^2
\leq \tilde{C}_0\int_X \left(1+\varrho_{\bar{x}}^2(u)\right)\mu_*(du)<\infty,
\end{align}
where $\tilde{C}_0$ is some positive constant. By the H\"{o}lder inequality, we get
\begin{align*}
E_{\mu_*}\left|\left(n^{-1/2}M_n\right)
\left(\chi(x_0)-\chi(x_n)\right)\right|
\leq\left(E_{\mu_*}\left(n^{-1}M_n^2\right)\right)^{1/2}
\left(E_{\mu_*}\left( \chi(x_0)-\chi(x_n)\right)^2\right)^{1/2}<\infty
\end{align*}
and therefore
\begin{align}\label{4}
\lim_{n\to\infty} 2n^{-1/2}E_{\mu_*}\left(\left(n^{-1/2}M_n\right)
\left(\chi(x_0)-\chi(x_n)\right)\right)=0.
\end{align}
Summarizing the above estimates (\ref{1})-(\ref{4}), we obtain
\begin{align*}
\lim_{n\to\infty} E_{\mu_*}\left(\left(S_n^*\right)^2\right)=E_{\mu_*}Z_1^2,
\end{align*}
which completes the proof.
\end{proof}

\begin{lemma}\label{lem1}
The square integrable martingale differences $(Z_n)_{n\geq 1}$ satisfy
\begin{align}\label{lem_sigma2}
\begin{aligned}
\frac{1}{n}\sum_{l=1}^nZ_l^2\to\sigma^2 \quad \text{Prob}_{\mu}\text{-a.s., as }n\to\infty,
\end{aligned}
\end{align}
and consequently condition $(\ref{th3})$ holds for $\sigma^2>0$.
\end{lemma}
\begin{proof}[Proof]
The idea of the proof is based on the property of asymptotic stability of the model, as well as on the Birkhoff Individual Ergodic Theorem.

The essence is to show that functions
\begin{align}\label{functions}
\begin{aligned}
x\mapsto E_x\Big(\Big|{\lim\inf}_{n\to\infty}\Big(\frac{1}{n}\sum_{l=1}^nZ_l^2\Big)-\sigma^2\Big|\wedge 1\Big)\\
x\mapsto E_x\Big(\Big|{\lim\sup}_{n\to\infty}\Big(\frac{1}{n}\sum_{l=1}^nZ_l^2\Big)-\sigma^2\Big|\wedge 1\Big)
\end{aligned}
\end{align}
are not only bounded, which is obvious, but also continuous. Indeed, if continuity is provided, we use the fact that $P_{\varepsilon}^m\mu$ converges weakly to $\mu_*$, as $m\to\infty$ (see Theorem 1, \cite{hhsw}), to obtain
\begin{align}\label{convergence}
\begin{aligned}
&E_{\mu}\Big(\Big|{\lim\inf}_{n\to\infty}\Big(\frac{1}{n}\sum_{l=1}^nZ_l^2\Big)-\sigma^2\Big|\wedge 1\Big)=\int_X E_x\Big(\Big|{\lim\inf}_{n\to\infty}\Big(\frac{1}{n}\sum_{l=1}^nZ_l^2\Big)-\sigma^2\Big|\wedge 1\Big)\mu(dx)\\
&=\int_XE_x\Big(\Big|{\lim\inf}_{n\to\infty}\Big(\frac{1}{n}\sum_{l=1}^nZ_l^2\Big)-\sigma^2\Big|\wedge 1\Big)P^m_{\varepsilon}\mu(dx)
\xrightarrow{m\to\infty} E_{\mu_*}\Big(\Big|{\lim\inf}_{n\to\infty}\Big(\frac{1}{n}\sum_{l=1}^nZ_l^2\Big)-\sigma^2\Big|\wedge 1\Big).
\end{aligned}
\end{align}

Now, if we compare it with the Birkhoff Individual Ergodic Theorem, which says that
\begin{align*}
\begin{aligned}
E_{\mu_*}\Big(\Big|{\lim\inf}_{n\to\infty}\Big(\frac{1}{n}\sum_{l=1}^nZ_l^2\Big)-\sigma^2\Big|\wedge 1\Big)=0,
\end{aligned}
\end{align*}
we may claim that
\begin{align*}
\begin{aligned}
E_{\mu}\Big(\Big|{\lim\inf}_{n\to\infty}\Big(\frac{1}{n}\sum_{l=1}^nZ_l^2\Big)-\sigma^2\Big|\wedge 1\Big)=0.
\end{aligned}
\end{align*}
This, in turn, impies
\begin{align}\label{estim:inf}
\begin{aligned}
{\lim\inf}_{n\to\infty}\Big(\frac{1}{n}\sum_{l=1}^nZ_l^2\Big)=\sigma^2\quad \text{Prob}_{\mu}\text{-a.s.}
\end{aligned}
\end{align}
Analogously, we may show that
\begin{align}\label{estim:sup}
\begin{aligned}
{\lim\sup}_{n\to\infty}\Big(\frac{1}{n}\sum_{l=1}^nZ_l^2\Big)=\sigma^2\quad \text{Prob}_{\mu}\text{-a.s.}
\end{aligned}
\end{align}
Finally, (\ref{estim:inf}) and (\ref{estim:sup}) imply $(\ref{lem_sigma2})$.

To complete the proof, continuity of both functions given by $(\ref{functions})$ should be established, just to make it clear that the convergence in $(\ref{convergence})$ occurs. Note that
\begin{align}\label{min_Hnk}
\begin{aligned}
E_x\left(\left|{\lim\inf}_{n\to\infty}
\left(\frac{1}{n}\sum_{l=1}^nZ_l^2\right)-\sigma^2\right|\wedge 1\right)
&=\lim_{n\to\infty}\lim_{k\to\infty}
E_x\left(\left|\min_{n\leq j\leq n+k}
\left(\frac{1}{j}\sum_{l=1}^jZ_l^2-\sigma^2\right)\right|\wedge 1\right)\\
&=\lim_{n\to\infty}\lim_{k\to\infty}H_{n,k}(x),
\end{aligned}
\end{align}
where
\begin{align}
\begin{aligned}
H_{n,k}(x):=E_x\left(\left|\min_{n\leq j\leq n+k}
\left(\frac{1}{j}\sum_{l=1}^jZ_l^2-\sigma^2\right)\right|\wedge 1\right).
\end{aligned}
\end{align}

Let us introduce
\begin{align}\label{def:g}
\begin{aligned}
\psi_{n,k}(y_0,\ldots,y_{n+k})&=
\left| \min_{n\leq j\leq n+k} 
\left( \frac{1}{j}
\left(\sum_{l=1}^{j}\left(\chi(y_l)-\chi(y_{l-1})+g(y_{l-1})\right)^2
\wedge j\left(1+\sigma^2\right)\right)
-\sigma^2 \right) \right|.
\end{aligned}
\end{align}

Recalling the definition of martingale differences $(Z_n)_{n\geq 1}$ (see (\ref{def:Z_n})) and following the property
\begin{align*}
\begin{aligned}
\left|\min_{n\leq j\leq n+k}\left(\frac{c_j}{j}-b\right)\right|\wedge 1
=\left|\min_{n\leq j\leq n+k}\left(\frac{1}{j}(c_j\wedge j(1+b))-b\right)\right|,
\end{aligned}
\end{align*}
we obtain
\begin{align}\label{H_as_g}
\begin{aligned}
H_{n,k}(x)=E_x\left(\psi_{n,k}(x_0, x_1, \ldots, x_{n+k})\right).
\end{aligned}
\end{align}

The idea to express the functions in interest in terms of $(\ref{min_Hnk})$-$(\ref{H_as_g})$ comes from \cite{bms} or \cite{kom_szar}. However, the final step to show the continuity of functions is established thanks to the coupling measure.

As mentioned at the beginning of the section $(x_n)_{n\geq 0}$ and $(y_n)_{n\geq 0}$ are Markov chains with transition probability function $\Pi_{\varepsilon}$ and initial distributions $\mu, \nu\in M_1^{2+\delta}(X)$, respectively. In particular, we may set $\mu:=\delta_x$ and $\nu:=\delta_y$. For technical reasons, we also consider $(\tilde{x}_n)_{n\geq 0}$ and $(\tilde{y}_n)_{n\geq 0}$, which are non-homogenous Markov chains with sequence of transition probability functions $(\Pi^1_{h_i})_{i\geq 1}$, given by (\ref{constr:n+1}), and initial distributions $\delta_x$ and $\delta_y$, respectively. Note that, according to (\ref{varepsilon_written_by_h}), we obtain
\begin{align*}
\begin{aligned}
P_{\varepsilon}\delta_x(\cdot)=\int_{\bar{B}(0,\varepsilon)}\Pi^1_{h_1}(x,\cdot)\nu^{\varepsilon}(dh_1)
\end{aligned}
\end{align*}
and therefore,
\begin{align}
\begin{aligned}
E_x\psi(x_0,\ldots,x_{n+k})=\int_{(\bar{B}(0,\varepsilon))^{n+k}}E_x\psi(\tilde{x}_0,\ldots,\tilde{x}_{n+k})\;\nu^{\varepsilon}(dh_1)\ldots\nu^{\varepsilon}(dh_{n+k}).
\end{aligned}
\end{align}
Let us remind the reader that there exists the appropriate coupling measure $C^{\infty}_{h_1,h_2,\ldots}((x,y),\cdot)$ on $(X^2)^{\infty}$ such that \[C^{\infty}_{h_1,h_2,\ldots}((x,y),A\times X)=\Pi^{\infty}_{h_1,h_2,\ldots}(x,A)\quad\text{and}\quad C^{\infty}_{h_1,h_2,\ldots}((x,y),X\times B)=\Pi^{\infty}_{h_1,h_2,\ldots}(y,B)\]
for every $A,B\in \otimes_{i=1}^{\infty}B_X$, as well as the coupling measure $\hat{C}^{\infty}_{h_1,h_2,\ldots}((x,y,1),\cdot)$ on the augmented space $(X^2\times\{0,1\})^{\infty}$ (see Section 7 in \cite{hhsw} for the full construction of coupling measures for iterated function systems). The expected value according to the measure $\hat{C}^{\infty}_{h_1,h_2,\ldots}((x,y,1),\cdot)$ is denoted by $E_{x,y}$.

Let us further introduce an auxiliary 
function 
\begin{align}\label{barH_as_psi}
\begin{aligned}
\tilde{H}_{n,k}(x)=E_x\Big(\psi_{n,k}(\tilde{x}_0, \tilde{x}_1, \ldots, \tilde{x}_{n+k})\Big).
\end{aligned}
\end{align}
Then,
\begin{align}\label{barH-barH}
\begin{aligned}
\lim_{n\to\infty}\lim_{k\to\infty}
\left|\tilde{H}_{n,k}(x)-\tilde{H}_{n,k}(y)\right|
&=\lim_{n\to\infty}\lim_{k\to\infty}
\left|E_x\left(\psi(\tilde{x}_0,\ldots,\tilde{x}_{n+k})\right)-E_y\left(\psi(\tilde{y}_0,\ldots,\tilde{y}_{n+k})\right)\right|\\
&\leq 
\lim_{n\to\infty}\lim_{k\to\infty}E_{x,y}
\left|\psi(\tilde{x}_0,\ldots,\tilde{x}_{n+k})-\psi(\tilde{y}_0,\ldots,\tilde{y}_{n+k})\right|.
\end{aligned}
\end{align}

It is easy to see that
\begin{align*}%\label{min-max}
\begin{aligned}
\min_{1\leq j\leq n}\left(f_j(x_j)\right)
-\min_{1\leq j\leq n}\left(f_j(y_j)\right)\leq
\max_{1\leq i\leq n}\left|f_i(x_i)-f_i(y_i)\right|
\end{aligned}
\end{align*}
for arbitrary functions $f_i:X\to R$ and points $x_i, y_i\in X$, $1\leq i\leq n$. 
We use this fact to obtain
\begin{align}\label{estim:min_na_max}
\begin{aligned}
&\lim_{n\to\infty}\lim_{k\to\infty}\left|\tilde{H}_{n,k}(x)-\tilde{H}_{n,k}(y)\right|\\
&\leq \lim_{n\to\infty}\lim_{k\to\infty}E_{x,y}
\Bigg(\max_{n\leq i\leq n+k}\frac{1}{i}\sum_{l=1}^{i}\Big|\left(\chi(\tilde{x}_l)-\chi(\tilde{x}_{l-1})+g(\tilde{x}_{l-1})\right)^2\wedge i(1+\sigma^2)\\
&\quad
-\left(\chi(\tilde{y}_l)-\chi(\tilde{y}_{l-1})+g(\tilde{y}_{l-1})\right)^2\wedge i(1+\sigma^2)\Big|\Bigg).
\end{aligned}
\end{align}

Note that the right side of (\ref{estim:min_na_max}) is equal to
\begin{align*}%\label{k0_average}
\begin{aligned}
&\lim_{n\to\infty}\lim_{k\to\infty}E_{x,y}
\Bigg(\max_{n\leq i\leq n+k}\frac{1}{i}\sum_{l=k_0}^{i}\Big|\left(\chi(\tilde{x}_l)-\chi(\tilde{x}_{l-1})+g(\tilde{x}_{l-1})\right)^2\wedge i(1+\sigma^2)\\
&\quad
-\left(\chi(\tilde{y}_l)-\chi(\tilde{y}_{l-1})+g(\tilde{y}_{l-1})\right)^2\wedge i(1+\sigma^2)\Big|\Bigg)
\end{aligned}
\end{align*}
for arbitrary $k_0\geq 1$. 
Further, due to the fact that the functions are bounded, we obtain
\begin{align}\label{estim:barH-barH_further}
\begin{aligned}
&\left|\tilde{H}_{n,k}(x)-\tilde{H}_{n,k}(y)\right|\\
&\leq 
E_{x,y}\left(
\max_{n\leq i\leq n+k}\frac{1}{i}\sum_{l=k_0}^{i}\left|\left(\chi(\tilde{x}_l)-\chi(\tilde{x}_{l-1})+g(\tilde{x}_{l-1})\right)-\left(\chi(\tilde{y}_l)-\chi(\tilde{y}_{l-1})+g(\tilde{y}_{l-1})\right)\right|2i(1+\sigma^2)\right)\\
&\leq 2(1+\sigma^2)E_{x,y}
\left(\sum_{l=k_0}^{n+k}|\chi(\tilde{x}_l)-\chi(\tilde{y}_l)|+|\chi(\tilde{x}_{l-1})-\chi(\tilde{y}_{l-1})|+|g(\tilde{x}_{l-1})-g(\tilde{y}_{l-1})|\right).
\end{aligned}
\end{align}
Let us now evaluate %the expression $E_{\Pi^*_{X^{2(n+k)}}\Pi^*_{1,\ldots,n+k}\hat{C}^{\infty}_{h_1,h_2\ldots}((x,y,1),\cdot)}|\chi(\bar{x}_l)-\chi(\bar{y}_l)|$
\begin{align*}
\begin{aligned}
&E_{x,y}\left|\chi(\tilde{x}_l)-\chi(\tilde{y}_l)\right|
\leq \sum_{i=0}^{\infty}E_{x,y}
\left|U_{\varepsilon}^ig(\tilde{x}_l)-U_{\varepsilon}^ig(\tilde{y}_l)\right|\\
&=\sum_{i=0}^{\infty}\int_{X^2}\left|U^i_{\varepsilon}g(u)-U^i_{\varepsilon}g(v)\right|\left(\Pi^*_{X^2}\Pi^*_l\hat{C}^{\infty}_{h_1,h_2,\ldots}((x,y,1),\cdot)\right)(du\times dv)\\
&=\sum_{i=0}^{\infty}\int_{X^2}
\left|\left\langle g,P^i_{\varepsilon}\delta_u\right\rangle
-\left\langle g,P^i_{\varepsilon}\delta_v\right\rangle\right|
\left(\Pi^*_{X^2}\Pi^*_l\hat{C}^{\infty}_{h_1,h_2,\ldots}((x,y,1),\cdot)\right)(du\times dv)\\
&=\sum_{i=0}^{\infty}\int_{X^2}
\left|\int_{\left(\bar{B}(0,\varepsilon)\right)^i}\int_{X^2}
g(z)\left(\Pi^i_{h_{l+1},\ldots,{h}_{l+i}}(u,\cdot)-\Pi^i_{{h}_{l+1},\ldots,{h}_{l+i}}(v,\cdot)\right)(dz)
\:\nu^{\varepsilon}(d{h}_{l+1})\ldots\nu^{\varepsilon}(d{h}_{l+i})\right|\\
&\quad\times
\left(\Pi^*_{X^2}\Pi^*_l\hat{C}^{\infty}_{h_1,h_2,\ldots}((x,y,1),\cdot)\right)(du\times dv)
\\
&\leq\sum_{i=0}^{\infty}\int_{\left(\bar{B}(0,\varepsilon)\right)^i}\int_{X^2}\int_{X^2}
\left|g(z_1)-g(z_2)\right|
\left(\Pi^*_{X^2}\Pi^*_i\hat{C}^{\infty}_{{h}_{l+1},{h}_{l+2},\ldots}((u,v,1),\cdot)\right)(dz_1\times dz_2)\\
&\quad\times\left(\Pi^*_{X^2}\Pi^*_l\hat{C}^{\infty}_{h_1,h_2,\ldots}((x,y,1),\cdot)\right)(du\times dv)
\:\nu^{\varepsilon}(d{h}_{l+1})\ldots\nu^{\varepsilon}(d{h}_{l+i}).
\end{aligned}
\end{align*}
The appropriate properties of coupling measure we use (see condition (d) in Section 3.2) imply the following estimation
\begin{align*}
\begin{aligned}
&E_{x,y}\left|\chi(\tilde{x}_l)-\chi(\tilde{y}_l)\right|
\\
&\leq
\sum_{i=0}^{\infty}\int_{\left(\bar{B}(0,\varepsilon)\right)^i}\int_{X^2}
\left|g(z_1)-g(z_2)\right|
\left(\Pi^*_{X^2}\Pi^*_{l+i}\hat{C}^{\infty}_{h_1,h_2,\ldots}((x,y,1),\cdot)\right)(dz_1\times dz_2)
\:\nu^{\varepsilon}(d{h}_{l+1})\ldots\nu^{\varepsilon}(d{h}_{l+i}).
\end{aligned}
\end{align*}
Hence, due to (\ref{coupling}), we obtain
\begin{align}\label{estim:chi(x_l)-chi(y_l)_bar}
\begin{aligned}
E_{x,y}\left|\chi(\tilde{x}_l)-\chi(\tilde{y}_l)\right|
&\leq \sum_{i=0}^{\infty}
\int_{\left(\bar{B}(0,\varepsilon)\right)^i}
CGq^{l+i}(1+\varrho_{\bar{x}}(x)+\varrho_{\bar{x}}(y))
\:\nu^{\varepsilon}(d{h}_{l+1})\ldots\nu^{\varepsilon}(d{h}_{l+i})\\
&\leq CG(1+\varrho_{\bar{x}}(x)+\varrho_{\bar{x}}(y))\sum_{i=l}^{\infty}q^i.
\end{aligned}
\end{align}
Simultaneously,
\begin{align}\label{estim:g-g}
\begin{aligned}
E_{x,y}|g(\tilde{x}_{l-1})-g(\tilde{y}_{l-1})|
\leq CGq^{l-1}(1+\varrho_{\bar{x}}(x)+\varrho_{\bar{x}}(y)).
\end{aligned}
\end{align}
Note that, thanks to (\ref{estim:chi(x_l)-chi(y_l)_bar}) and (\ref{estim:g-g}), the expression (\ref{estim:barH-barH_further}) may now be estimated by 
\begin{align*}
\begin{aligned}
&2(1+\sigma^2)CG(1+\varrho_{\bar{x}}(x)+\varrho_{\bar{x}}(y))\sum_{l=k_0}^{n+k}
\left(\sum_{i=l}^{\infty}q^i+\sum_{i=l-1}^{\infty}q^i+q^{l-1}\right)\\
&\qquad\qquad\qquad\qquad\leq 4(1+\sigma^2)CG(1+\varrho_{\bar{x}}(x)+\varrho_{\bar{x}}(y))\sum_{l=k_0}^{n+k}\left(q^{l-1}\sum_{i=0}^{\infty}q^i\right)\\
&\qquad\qquad\qquad\qquad=\frac{4}{1-q}(1+\sigma^2)CG(1+\varrho_{\bar{x}}(x)+\varrho_{\bar{x}}(y))\sum_{l=k_0}^{n+k}q^{l-1}.
\end{aligned}
\end{align*}
The estimate is independent of $(h_i)_{i\geq 1}$ and therefore we obtain
\begin{align*}%\label{estim:H-H}
\begin{aligned}
\lim_{n\to\infty}\lim_{k\to\infty}|H_{n,k}(x)-H_{n,k}(y)|\leq \frac{4}{1-q}(1+\sigma^2)CG(1+\varrho_{\bar{x}}(x)+\varrho_{\bar{x}}(y))\sum_{l=k_0}^{\infty}q^{l-1}.
\end{aligned}
\end{align*}
Note that $k_0$ is arbitrary and therefore can be chosen so small that $\sum_{l=k_0}^{\infty}q^{l-1}$ is as close to zero as we wish. Then, $\lim_{n\to\infty}\lim_{k\to\infty}|H_{n,k}(x)-H_{n,k}(y)|=0$ for every $x,y\in X$. The proof is complete.

%To sum up, for every $x,y\in X$ and every $\epsilon>0$, we may choose such $k_0>0$ that $\lim_{n,k\to\infty}|H_{n,k}(x)-H_{n,k}(y)|\leq \epsilon$, which completes the proof.

\end{proof}

\begin{lemma}\label{lem2}
Let $\sigma^2>0$. Under Assumptions (I)-(VI), the square integrable martingale differences $(Z_n)_{n\geq 1}$ satisfy conditions (\ref{th1}) and (\ref{th2}).
\end{lemma}

\begin{proof}[Proof]
Let $\mu\in M_1^{2+\delta}(X)$. Note that 
\begin{align*}
\begin{aligned}
\sum_{n=1}^{\infty}s_n^{-4}E_{\mu}\Big(Z_n^41_{\{|Z_n|<\gamma s_n\}}\Big)&\leq\sum_{n=1}^{\infty}s_n^{-4}\gamma^{2-\delta}s_n^{2-\delta}E_{\mu}|Z_n|^{2+\delta}\leq \gamma^{2-\delta}\sup_{n\geq 1}E_{\mu}|Z_n|^{2+\delta}\sum_{n=1}^{\infty}s_n^{-2-\delta}.
\end{aligned}
\end{align*}
Recall that $\sup_{n\geq 1}E_{\mu}|Z_n|^{2+\delta}<\infty$ (see (\ref{supEZn-delta-finite})). 
On the other hand, by Lemma \ref{lem:sn/n=sigma} we have $s_n^2/n\to\sigma^2$, as $n\to\infty$, which implies $\sum_{n=1}^{\infty}s_n^{-2-\delta}<\infty$ and completes the proof of condition (\ref{th1}).

To show condition (\ref{th2}), observe that 
\begin{align*}
\begin{aligned}
\sum_{n=1}^{\infty}s_n^{-1}E_{\mu}\left(|Z_n|1_{\{|Z_n|\geq \vartheta s_n\}}\right)&\leq\sum_{n=1}^{\infty}s_n^{-1}E_{\mu}\left(\frac{|Z_n|^{2+\delta}}{(\vartheta s_n)^{1+\delta}}\right)\leq \vartheta^{-1-\delta}\sup_{n\geq 1}E_{\mu}|Z_n|^{2+\delta}\sum_{n=1}^{\infty}s_n^{-2-\delta}<\infty.
\end{aligned}
\end{align*}
\end{proof}

\subsection{Main result}

The CLT is verified for the generalised cell cycle model introduced and characterised in Section \ref{sec:ass+prop} (see Theorem 2, \cite{hhsw}). Now, it is natural to ask for the proof of the LIL.

\begin{thm}\label{lil}
Let $(X,\varrho)$ be a Polish space and $(x_n)_{n\geq 0}$ a Markov chain with state space $X$, transition probability function $\Pi_{\varepsilon}$ and initial distribution $\mu\in M^{2+\delta}_1(X)$. We assume conditions (I)-(VI), which entail all properties described in Section \ref{sec:ass+prop}. If $g$ is a Lipschitz function with $\langle g,\mu_*\rangle =0$ and $\sigma^2>0$, then $\text{Prob}_{\mu}$-a.s. the sequence 
\begin{align*}
\begin{aligned}
\theta_n(t)=\frac{\sum_{i=1}^kg(x_i)+(nt-k)g(x_{k+1})}{\sigma\sqrt{2n \log\log n}}
\end{aligned}
\end{align*}
for $k\leq nt\leq k+1$, $k=1,\ldots,n-1$, $t>0$, $n>e$, and $\theta_n(t)=0$ otherwise, is relatively compact in $\mathcal{C}$ and the set of its limit points coincides with $\mathcal{K}$.
\end{thm}

\begin{proof}[Proof]
We begin with the observation that, since $s_n^2/n\to\sigma^2>0$, as $n\to\infty$ (see Lemma \ref{lem:sn/n=sigma}), we obtain
\begin{align*}
\begin{aligned}
\frac{\sqrt{2s^2_n\log\log s_n^2}}{\sigma\sqrt{2n\log\log n}}\to 1,\quad\text{as }n\to\infty.
\end{aligned}
\end{align*}
Hence, from Lemmas \ref{lem1} and \ref{lem2}, it follows that the sequence
\begin{align*}
\begin{aligned}
\eta_n(t)=\frac{M_k+(s_n^2t-s_k^2)(s_{k+1}^2-s_k^2)^{-1}Z_{k+1}}{\sigma\sqrt{2n\log\log n}}
\end{aligned}
\end{align*}
for $s_k^2\leq s_n^2t\leq s_{k+1}^2$, $k=1,\ldots, n-1$ and $t>0$, $n>e$, and $\eta_n(t)=0$ otherwise, is relatively compact in $\mathcal{C}$ and the set of its limit points coincides with $\mathcal{K}$ (see \cite{heyde_scott}). Let $t\in(0,1]$ and $n\geq e$. Now, if $k\leq nt\leq k+1$, then
\begin{align*}
\begin{aligned}
\frac{k\sigma^2}{s_k^2}s_k^2\leq \frac{n\sigma^2}{s_n^2}ts_n^2\leq\frac{(k+1)\sigma^2}{s_{k+1}^2}s_{k+1}^2.
\end{aligned}
\end{align*}
Set
\begin{align}
\begin{aligned}
\hat{\eta}_n(t):=\frac{M_k+(nt-k)Z_{k+1}}{\sigma\sqrt{2n\log\log n}},
\end{aligned}
\end{align}
where $k\geq 1$ is such that $k\leq nt\leq k+1$. Since $n\sigma^2/s_n^2\to 1$, as $n\to\infty$, we obtain
\begin{align*}
\begin{aligned}
(1-\tilde{\epsilon})s_k^2\leq (1+\tilde{\epsilon})s_n^2 t\leq (1+\tilde{\epsilon})^2(1-\tilde{\epsilon})^{-1}s_{k+1}^2
\end{aligned}
\end{align*}
for every $\tilde{\epsilon}>0$ and $n$ large enough. 
As a consequence, there is $t_*\in[t(1-\tilde{\epsilon})(1+\tilde{\epsilon})^{-1},t(1+\tilde{\epsilon})(1-\tilde{\epsilon})^{-1}]$ such that $s_k^2\leq s_n^2t_*\leq s_{k+1}^2$. On the other hand, the diameters of the intervals $[s_k^2/s_n^2,s_{k+1}^2/s_n^2]$ for $1\leq k\leq n-1$, converge to $0$, as $n\to\infty$. Hence, there exists $t_n>0$ such that $\hat{\eta}_n(t)=\eta_n(t_n)$ and $t_n\to t$, as $n\to\infty$. Recall that the sequence $(\eta_n(t))_{n\geq e}$ is relatively compact in $\mathcal{C}$ and the set of its limit points coincides with $\mathcal{K}$. Therefore, the sequence $(\hat{\eta}_n(t))_{n>e}$ is also relatively compact in $\mathcal{C}$ and has the same set of limit points.

Fix $\tilde{\epsilon}>0$ and define the set
\begin{align}
\begin{aligned}
A_n:=\left\{\omega\in\Omega:\;\frac{|M_n(\omega)-\sum_{i=1}^{n-1}g(x_i(\omega))|}{\sqrt{n}}\geq \frac{\tilde{\epsilon}}{2}\right\}\cup\left\{\omega\in\Omega:\;\frac{|Z_n(\omega)-g(x_n(\omega))|}{\sqrt{n}}\geq\frac{\tilde{\epsilon}}{2}\right\}\;\text{for }n\geq 1.
\end{aligned}
\end{align}
Let us now show that $\sum_{n=1}^{\infty}\text{Prob}_{\mu}(A_n)<\infty$. Choose $\epsilon>0$. Note that, by the Markov inequality and the fact that there is $\delta_1>0$ such that $(\zeta+\xi)^{2+\delta}\leq(2+\epsilon)(\zeta^{2+\delta}+\xi^{2+\delta})$ for $\zeta,\xi\geq 0$, $\delta\in(0,\delta_1)$, we obtain
\begin{align*}
\begin{aligned}
\text{Prob}_{\mu}\left(\omega\in\Omega:\;\frac{|M_n(\omega)-\sum_{i=1}^{n-1}g(x_i(\omega))|}{\sqrt{n}}\geq\frac{\tilde{\epsilon}}{2}\right)
&=\text{Prob}_{\mu}\left(\omega\in\Omega:\;\frac{|\chi(x_n(\omega))-\chi(x_0(\omega))|}{\sqrt{n}}\geq\frac{\tilde{\epsilon}}{2}\right)\\
&\leq \left(\frac{2}{\tilde{\epsilon}}\right)^{2+\delta}(2+\epsilon)\frac{E_{\mu}|\chi(x_n)|^{2+\delta}+E_{\mu}|\chi(x_0)|^{2+\delta}}{n^{1+\delta/2}}.
\end{aligned}
\end{align*}
Now, observe that, due to Lemma \ref{lem:chi-chi}
\begin{align}\label{estim:E|chi|2delta}
\begin{aligned}
E_{\mu}|\chi(x_n)|^{2+\delta}&=\int_X|\chi(x)|^{2+\delta}P^n_{\varepsilon}\mu(dx)\\
&\leq (2+\epsilon)|\chi(\bar{x})|^{2+\delta}+(2+\epsilon)\int_X|\chi(x)-\chi(\bar{x})|^{2+\delta}P^n_{\varepsilon}\mu(dx)\\
&\leq (2+\epsilon)|\chi(\bar{x})|^{2+\delta}+(2+\epsilon)^2G^{2+\delta}C^{2+\delta}(1-q)^{-(2+\delta)}(1+\langle \varrho_{\bar{x}}^{2+\delta},P^n_{\varepsilon}\mu\rangle ).
\end{aligned}
\end{align}
Note that the first component of (\ref{estim:E|chi|2delta}) is finite due to (\ref{fact:finite_chi}) and (\ref{def:chi}), while the second is finite due to Lemma \ref{H3}. Hence,
\begin{align}\label{estim:An1}
\begin{aligned}
\text{Prob}_{\mu}\left(\omega\in\Omega:\;\frac{|M_n(\omega)-\sum_{i=1}^{n-1}g(x_i(\omega))|}{\sqrt{n}}\geq\frac{\tilde{\epsilon}}{2}\right)
\leq \frac{c_1}{n^{1+\delta/2}},
\end{aligned}
\end{align}
where $c_1>$ is some constant independent of $n$. Similarly,
\begin{align}\label{estim:An2}
\begin{aligned}
\text{Prob}_{\mu}\left\{\omega\in\Omega:\;\frac{|Z_n(\omega)-g(x_n(\omega))|}{\sqrt{n}}\geq\frac{\tilde{\epsilon}}{2}\right\}
&=\text{Prob}_{\mu}\left\{\omega\in\Omega:\;\frac{|\chi(x_{n+1}(\omega))-\chi(x_{n}(\omega))|}{\sqrt{n}}\geq\frac{\tilde{\epsilon}}{2}\right\}\\
&\leq \frac{c_2}{n^{1+\delta/2}},
\end{aligned}
\end{align}
where $c_2>$ is some constant independent of $n$. By (\ref{estim:An1}) and (\ref{estim:An2}), the series $\sum_{n=1}^{\infty}\text{Prob}_{\mu}(A_n)$ is convergent. 

Finally, following the Borel-Cantelli Lemma, we obtain that $\text{Prob}_{\mu}$-a.s.
\begin{align*}
\begin{aligned}
{\lim\sup}_{n\to\infty}{\sup}_{0< t\leq 1}\left|\frac{M_k-(nt-k)Z_{k+1}}{\sigma\sqrt{2n\log\log n}}-\frac{\sum_{i=1}^kg(x_i)+(nt-k)g(x_{k+1})}{\sigma\sqrt{2n\log\log n}}\right|<\tilde{\epsilon},
\end{aligned}
\end{align*}
where $k\leq nt\leq k+1$. This implies ${\lim\sup}_{n\to\infty}\sup_{0<t\leq 1} |\tilde{\eta}_n(t)-\theta_n(t)|\leq \tilde{\epsilon}$. Since $\tilde{\epsilon}>0$ was arbitrary, the proof is complete.
\end{proof}

\end{document}